\documentclass[11pt,francais]{smfart}

\usepackage[T1]{fontenc}
\usepackage[francais]{babel}

\usepackage{amssymb,url,xspace,smfthm,amsmath}

\DeclareFontFamily{U}{wncy}{}
\DeclareFontShape{U}{wncy}{m}{n}{
<5>wncyr5
<6>wncyr6
<7>wncyr7
<8>wncyr8
<9>wncyr9
<10>wncyr10
<11>wncyr10
<12>wncyr6
<14>wncyr7
<17>wncyr8
<20>wncyr10
<25>wncyr10}{}
\DeclareMathAlphabet{\cyr}{U}{wncy}{m}{n}

\newcommand{\BibTeX}{{\scshape Bib}\kern-.08em\TeX}
\newcommand{\T}{\S\kern .15em\relax }
\newcommand{\AMS}{$\mathcal{A}$\kern-.1667em\lower.5ex\hbox
        {$\mathcal{M}$}\kern-.125em$\mathcal{S}$}

\usepackage{tikz}
\usetikzlibrary{matrix,positioning}
\title[Fibrations en torseurs sous un tore]{Principes locaux-globaux pour certaines fibrations en torseurs sous un tore}
\date {2014}
\author{Arne Smeets}

\address{Departement Wiskunde, KU Leuven, Celestijnenlaan 200B, 3001
  Leuven, Belgium \emph{et} D\'epartement de Math\'ematiques,
  B\^atiment 425, Universit\'e Paris-Sud 11, 91405 Orsay, France}
\email{arnesmeets@gmail.com}
\urladdr{}

\begin{document}
\def\smfbyname{}

\begin{abstract} Soit $k$ un corps de nombres et soit $T$ un $k$-tore. Consid\'erons une fibration en torseurs sous $T$, c'est-\`a-dire un morphisme $f: X \to \mathbb{P}^1_k$ d'une $k$-vari\'et\'e projective et lisse $X$ vers $\mathbb{P}^1_k$ dont la fibre g\'en\'erique $X_\eta \to \eta$ est une compactification lisse d'un espace principal homog\`ene sous $T \otimes_k \eta$. On \'etudie dans ce texte l'obstruction de Brauer-Manin au principe de Hasse et \`a l'approximation faible pour $X$, sous l'hypoth\`ese de Schinzel. On g\'en\'eralise les r\'esultats r\'ecents de Wei \cite{Wei11}. Nos r\'esultats sont inconditionnels si $k = \mathbf{Q}$ et les fibres non-scind\'ees de $f$ sont d\'efinies sur $\mathbf{Q}$. On \'etablit \'egalement un analogue inconditionnel de notre r\'esultat principal pour les z\'ero-cycles de degr\'e $1$. \vskip 10pt \noindent \textbf{\emph{Summary.} ---}  Let $k$ be a number field and $T$ a $k$-torus. Consider a family of torsors under $T$, i.e. a morphism $f: X \to \mathbb{P}^1_k$ from a projective, smooth $k$-variety $X$ to $\mathbb{P}^1_k$, the generic fibre $X_\eta \to \eta$ of which is a smooth compactification of a principal homogeneous space under $T \otimes_k \eta$. We study the Brauer--Manin obstruction to the Hasse principle and to weak approximation for $X$, assuming Schinzel's hypothesis. We generalize Wei's recent results \cite{Wei11}. Our results are unconditional if $k = \mathbf{Q}$ and all non-split fibres of $f$ are defined over $\mathbf{Q}$. We also establish an unconditional analogue of our main result for zero-cycles of degree $1$. 
\end{abstract}
\maketitle


\section{Introduction}

\subsection{Historique et r\'esultats} Fixons un corps de nombres $k$, une cl\^oture alg\'ebrique $\overline{k}$ de $k$ et une extension ab\'elienne maximale $k^\mathrm{ab} \subseteq \overline{k}$ de $k$. Pour les notions de base sur le principe de Hasse, l'approximation faible, le groupe de Brauer et l'obstruction de Brauer-Manin, voir par exemple \cite[\S 3]{CTSan87a}.

On s'int\'eresse dans cet article aux principes locaux-globaux pour certaines vari\'et\'es projectives et lisses $X$, d\'efinies sur le corps de nombres $k$ et fibr\'ees sur la droite projective, c'est-\`a-dire munies d'un morphisme (dominant) $f: X \to \mathbb{P}^1_k$. 

La question si l'obstruction de Brauer-Manin contr\^ole la validit\'e du principe de Hasse et l'approximation faible pour l'espace total d'une telle fibration $f: X \to \mathbb{P}^1_k$, s'il en est ainsi pour les fibres de $f$, a \'et\'e \'etudi\'ee par de nombreux auteurs. Citons par exemple les articles classiques de Harari \cite{Har}, Colliot-Th\'el\`ene et Swinnerton-Dyer \cite{CTSD94} et Colliot-Th\'el\`ene, Skorobogatov et Swinnerton-Dyer \cite{CTSkoSD98}. Ce probl\`eme est  difficile: l'hypoth\`ese de Schinzel \cite[\S 4]{CTSD94} est devenu un outil standard pour traiter le cas d'une fibration avec plusieurs fibres non-scind\'ees. Or, si on ne suppose pas que le principe de Hasse et l'approximation faible valent pour les fibres de $f$ (comme dans \cite{CTSkoSD98}) mais seulement que l'obstruction de Brauer-Manin est la seule dans les fibres, tr\`es peu est connu sur la situation. On s'int\'eresse dans ce texte \`a une classe particuli\`ere de telles fibrations:

\begin{defi} Soit $T$ un $k$-tore. On d\'efinit une \emph{fibration en torseurs sous $T$} comme une $k$-vari\'et\'e projective, lisse et g\'eom\'etriquement connexe $X$, munie d'un morphisme dominant $f: X \to \mathbb{P}^1$ tel que la fibre g\'en\'erique $X_\eta \to \eta$ de $f$ soit une compactification lisse d'un espace principal homog\`ene sous le $\eta$-tore $T \times_k \eta$.  \end{defi}

Le cas particulier d'un tore dit ``normique'' $T = R^1_{K/k}\mathbb{G}_m$ associ\'e \`a une extension finie $K/k$ a \'et\'e \'etudi\'e de fa\c{c}on tr\`es intensive. Ce cas correspond \`a la classe des vari\'et\'es qui sont des compactifications lisses d'une vari\'et\'e affine donn\'ee par une \'equation de la forme $$N_{K/k}(\mathbf{x}) = P(t) \neq 0,$$ o\`u $P(t) \in k[t]$ est un polyn\^ome non constant. Citons les articles r\'ecents \cite{BHB11}, \cite{DSW}, \cite{Liang12}, \cite{SchSk}, \cite{VV} et \cite{Wei11}. Le spectre des techniques employ\'ees est large: citons la m\'ethode de la descente (\cite{CTHarSko03}, \cite{DSW}), les techniques de la th\'eorie analytique des nombres (\cite{BHB11}), la m\'ethode des fibrations (\cite{Liang12}, \cite{Wei11}) ou une combinaison de celles-ci (\cite{CTHarSko03}, \cite{HBSko}, \cite{SchSk}).

On g\'en\'eralise ici les r\'esultats de Wei \cite{Wei11} obtenus  via la m\'ethode des fibrations (sous l'hypoth\`ese de Schinzel), tout en donnant des preuves plus simples. Wei a \'etudi\'e l'obstruction de Brauer-Manin pour les \'equations ``normiques'' $N_{K/k}(\mathbf{x}) = P(t)$, avec $K/k$ une extension \emph{ab\'elienne}. Son r\'esultat principal \cite[Theorem 3.1]{Wei11} dit que le principe de Hasse et l'approximation faible valent pour toute compactification lisse d'une telle vari\'et\'e (sous l'hypoth\`ese de Schinzel), d\`es que l'obstruction de Brauer-Manin ne s'y oppose pas et que la condition $\cyr{X}^2_\omega(\widehat{T}) = \cyr{X}^2_{\omega}(\widehat{T})_P$ est satisfaite (voir \cite[D\'efinition 2.4]{CTHarSko03}). Comme il le remarque \cite[\S 3]{Wei11}, cette derni\`ere condition est \'equivalente \`a la surjectivit\'e de la fl\`eche naturelle $$\mathrm{Br}(X^{\mathsf{CTHS}})/\mathrm{Br}(k) \longrightarrow \mathrm{Br}(X_\eta)/\mathrm{Br}(\eta),$$ o\`u $X^\mathsf{CTHS}$ est une certaine compactification partielle de la vari\'et\'e affine d\'efinie par l'\'equation $N_{K/k}(\mathbf{x}) = P(t) \neq 0$, construite dans \cite[\S 2]{CTHarSko03}. 

On g\'en\'eralise son r\'esultat dans plusieurs directions, dans le sens qu'on traite des tores g\'en\'eraux au lieu des tores ``normiques'', et qu'on impose une condition moins forte que sa condition $\cyr{X}^2_\omega(\widehat{T}) = \cyr{X}^2_{\omega}(\widehat{T})_P$. 

Le r\'esultat principal de ce texte -- qui englobe \'egalement \cite[Theorem 3.2]{Wei11} -- est le th\'eor\`eme suivant.

\begin{theo}
\label{principal} Soit $T$ un $k$-tore quasi-d\'eploy\'e par $k^{\mathrm{ab}}$. Soit $f: X \to \mathbb{P}^1_k$ une fibration en torseurs sous $T$. Alors le groupe $\mathrm{Br}(X_\eta)/\mathrm{Br}(\eta)$ est ab\'elien fini. Faisons l'hypoth\`ese qu'il existe un ouvert $X^\circ \subseteq X$ lisse, surjectif et g\'en\'eriquement projectif \`a fibres connexes au-dessus de $\mathbb{A}^1_k$, de fa\c{c}on que le morphisme $X^\circ \to \mathbb{A}^1_k$ admette une section sur $k^{\mathrm{ab}}$ et que $\mathrm{Br}(X_\eta)/\mathrm{Br}(\eta)$ soit engendr\'e par les images des groupes suivants:
\begin{itemize}
\item le groupe $\mathrm{Br}(X^\circ)$, et
\item le groupe $\mathrm{Ker}[\mathrm{Br}(X_\eta) \to \mathrm{Br}(X_\eta \otimes_k k^\mathrm{ab})]$.
\end{itemize}

Admettons l'hypoth\`ese de Schinzel. Alors $X(k)$ est dense dans $X(\mathbf{A}_k)^{\mathrm{Br}(X)}$. \end{theo}

Rappelons qu'une extension de corps $K/k$ \emph{quasi-d\'eploie} le tore $T$ si le $K$-tore $T \otimes_k K$ est quasi-trivial, i.e. produit de facteurs de la forme $R_{L/K}\mathbb{G}_{m,L}$, pour $L/K$ une extension finie (et variable). 

\begin{rema} \label{addcomb} Si $k = \mathbf{Q}$ et si toutes les fibres non-scind\'ees de $f$ -- voir \cite[\S 1]{Sko} pour cette notion -- sont d\'efinies sur $\mathbf{Q}$, on peut d\'emontrer le Th\'eor\`eme \ref{principal} de fa\c{c}on inconditionnelle, c'est-\`a-dire sans admettre l'hypoth\`ese de Schinzel. On verra ci-dessous que cela r\'esulte des travaux r\'ecents de Green, Tao et Ziegler en combinatoire additive, dans la forme pr\'esent\'ee dans \cite{HSW}. Au \S 5, on \'etablira l'analogue inconditionnel du Th\'eor\`eme \ref{principal} pour les z\'ero-cycles de degr\'e $1$ (voir Th\'eor\`eme \ref{0cyc}). \end{rema}
 
L'hypoth\`ese faite sur le groupe de Brauer de la fibre g\'en\'erique dans l'\'enonc\'e principal est quelque peu technique; il n'est pas \'evident de la v\'erifier dans des cas concrets. Or, il se trouve que cette hypoth\`ese est toujours satisfaite si on suppose qu'il existe une extension \emph{cyclique} de $k$ qui quasi-d\'eploie le tore $T$. Ceci donne le corollaire suivant, qui sera d\'emontr\'e \`a la fin du \S 4:

\begin{coro} \label{coroll} Soit $T$ un $k$-tore qui est quasi-d\'eploy\'e par une extension cyclique. Soit $f: X \to \mathbb{P}^1_k$ une fibration en torseurs sous $T$. Admettons l'hypoth\`ese de Schinzel. Alors $X(k)$ est dense dans $X(\mathbf{A}_k)^{\mathrm{Br}(X)}$.
\end{coro}

En particulier:

\begin{coro} Soit $X$ un mod\`ele projectif et lisse de la $k$-vari\'et\'e affine d\'efinie par $$\prod_{i = 1}^n N_{K_i/k}(\mathbf{x}_i) = P(t) \neq 0,$$ o\`u $P(t) \in k[t]$ est quelconque et les corps $K_1,\cdots,K_n$ sont des extensions finies du corps de nombres $k$ dont au moins une est une extension \emph{cyclique}. Admettons l'hypoth\`ese de Schinzel. Alors $X(k)$ est dense dans $X(\mathbf{A}_k)^{\mathrm{Br}(X)}$. \end{coro}
  
  Si $k = \mathbf{Q}$ et si le polyn\^ome $P(t)$ a toutes ses racines rationnelles, le r\'esultat ci-dessus est inconditionnel, comme on l'a remarqu\'e plus haut.

\smallskip

Finalement, au \S 6 (qui est ind\'ependant du reste du texte), on fait une remarque sur l'hypoth\`ese de Schinzel qui nous permet de l'utiliser de fa\c{c}on plus souple. Ceci m\`enera \`a une g\'en\'eralisation de \cite[Theorem 1.1.e]{CTSkoSD98}.

\subsection{Remerciements} Je voudrais r\'emercier Jean-Louis Colliot-Th\'el\`ene pour son aide, ses encouragements et sa patience pendant l'\'elaboration de ce texte. Je remercie \'egalement Olivier Wittenberg pour des discussions utiles, notamment sur la Proposition \ref{HilbertSchinzel}.
 
L'auteur est aspirant du FWO (Fonds de la Recherche Scientifique - Flandres). Pendant la pr\'eparation de ce texte, il a \'egalement b\'en\'efici\'e du soutien du Centre Interfacultaire Bernoulli (EPFL) et du Hausdorff  Institute for Mathematics (Bonn).

\section{G\'en\'eralit\'es sur les fibrations en torseurs sous un tore}

Soit $k$ un corps de caract\'eristique z\'ero $k$ et soit $T$ un $k$-tore quelconque. Choisissons une compactification lisse et \'equivariante $T^c$ de $T$. 

Soit $f: X \to \mathbb{P}^1_k$ une fibration en torseurs sous le tore $T$. Il existe alors un ouvert $U \subseteq \mathbb{P}^1_k$ et un ouvert $V \subseteq X_U = f^{-1}(U)$ de fa\c{c}on que $V$ soit un espace principal homog\`ene sur $U$ sous $T$. Le produit contract\'e $V \times^T T^c$ est alors une $k$-vari\'et\'e projective et lisse au-dessus de $U$ qui contient $V$ comme ouvert. Des arguments classiques impliquent que le morphisme $V \times^T T^c \to U$ s'\'etend en un morphisme $g: Y \to \mathbb{P}^1_k$ qui satisfait $g^{-1}(U) = V \times^T T^c$, o\`u $Y$ est une $k$-vari\'et\'e projective et lisse $k$-birationnelle \`a $X$. 

Or, les \'enonc\'es que l'on veut \'etablir dans ce texte ne d\'ependent que de la classe d'\'equivalence birationnelle de la vari\'et\'e projective et lisse $X$. On suppose donc jusqu'\`a la fin du \S 3 que la $k$-vari\'et\'e $X$ est une fibration en torseurs sous $T$ telle qu'au-dessus d'un ouvert convenable $U \subseteq \mathbb{P}^1_k$, l'image inverse $X_U = f^{-1}(U)$ est de la forme $V \times^T T^c$, o\`u $V$ est un $U$-torseur sous $T$ ($\star$).

\begin{lemm} \label{spec} Soit $k$ un corps de nombres. Pour tout $Q \in U(k)$, on dispose d'une fl\`eche de sp\'ecialisation $$\mathrm{Br}(X_\eta)/\mathrm{Br}(\eta) \longrightarrow \mathrm{Br}(X_Q)/\mathrm{Br}(k)$$ qui est un isomorphisme. \end{lemm}
\begin{proof}[D\'emonstration] On fixe un point \`a l'infini $\infty \in \mathbf{P}^1_k$. Soit $t$ une coordonn\'ee affine sur $\mathrm{Spec}\,k[t] = \mathbf{A}^1_k = \mathbf{P}^1_k \setminus \{\infty\}$. Soit $R = \mathrm{Spec}\,\overline{k}[t]_{(t - \lambda)}$, o\`u $\lambda$ est la $t$-coordonn\'ee de $Q \in U(k)$. Consid\'erons le $R$-sch\'ema $X_R$; sa fibre g\'en\'erique est $X_\eta \otimes_k \overline{k}$. Par \cite[Lemme 3.1.1]{Har}, on a un isomorphisme $$\mathrm{Pic}(X_R) \cong \mathrm{Pic}(X_\eta \otimes_k \overline{k}).$$ La fibre sp\'eciale de $X_R$ est $\overline{X_Q}$; on obtient donc une fl\`eche $$\mathrm{Pic}(X_\eta \otimes_k \overline{k}) \to \mathrm{Pic}(\overline{X_Q})$$ qui est un isomorphisme Galois-\'equivariant \cite[Lemme 2.1]{CTHarSko03}. Il en r\'esulte que $$H^1(k,\mathrm{Pic}(X_\eta \otimes_k \overline{k})) \cong H^1(k,\mathrm{Pic}(\overline{X}_Q)) \cong \mathrm{Br}(X_Q)/\mathrm{Br}(k).$$ Pour finir la preuve, il suffit maintenant de remarquer que $$H^1(k,\mathrm{Pic}(X_\eta \otimes_k \overline{k})) \cong H^1(\eta,\mathrm{Pic}(X_{\overline{\eta}})) \cong \mathrm{Br}(X_\eta)/\mathrm{Br}(\eta).$$ (Notons que ce dernier \'enonc\'e  utilise $H^3(k(t),\mathbb{G}_m) = 0$, pour $k$ un corps de nombres.) \end{proof}

La proposition suivante se montre utile pour l'\'etude des fibres non scind\'ees d'une fibration en torseurs sous un tore. Il implique en particulier que si le $k$-tore $T$ est d\'eploy\'e (ou quasi-d\'eploy\'e) par une extension ab\'elienne de $k$, alors la fibration satisfait \`a la condition d'ab\'elianit\'e de \cite[Theorem 1.1]{CTSkoSD98}. Pour la d\'efinition de tores flasques, on renvoie \`a \cite[Definition 1.2]{CTSan87b}.

\begin{prop} \label{inutile} On suppose qu'il existe une extension finie $K$ du corps $k$ telle que le $K$-tore $T \otimes_k K$ soit flasque. Soit $Q \in \mathbb{P}^1_k$ un point ferm\'e de corps r\'esiduel $L$. 

Alors $X_Q$ admet une composante irr\'eductible $Y$ de multiplicit\'e $1$ telle que la fermeture alg\'ebrique de $L$ dans le corps de fonctions $L(Y)$ soit un sous-corps d'un corps qui est facteur direct de $K \otimes_k L$.  \end{prop} 

\begin{proof}[D\'emonstration] La $K$-vari\'et\'e $X_U \otimes_k K$ est un torseur sur $U \otimes_k K$ sous $T \otimes_k K$ et s'\'etend en un torseur sur $\mathbb{P}^1_K$ tout entier gr\^ace \`a \cite[Theorem 2.2.(i)]{CTSan87b}. Si $R$ est un point ferm\'e de $\mathbb{P}^1_K$ au-dessus de $Q$, alors \cite[Proposition 3.4.(b)]{CIME} implique que la fibre $(X \otimes_k K)_R$ a une composante g\'eom\'etriquement int\`egre de multiplicit\'e $1$. On en d\'eduit facilement le r\'esultat voulu. \end{proof}

Soit $X^\circ \subseteq X$ un ouvert lisse, surjectif et g\'en\'eriquement projectif \`a fibres connexes sur la droite affine $\mathbb{A}^1_k$, tel que le morphisme induit $X^\circ \to \mathbb{A}^1_k$ admet une section sur $\overline{k}$. Un tel ouvert $X^\circ$ existe toujours: si $L$ est une extension de $k$ qui d\'eploie le tore $T$, alors la fibre g\'en\'erique $X_\eta \otimes_k L$ est un torseur sous le tore d\'eploy\'e $(T \times_k \eta) \otimes_k L$. Comme tout espace principal homog\`ene sous un tore trivial est trivial, on voit que $X_\eta \otimes_k L$ admet un $(\eta \otimes_k L)$-point qui s'\'etend en une section $\mathbb{P}^1_L \to X \otimes_k L$ dont l'image est contenue dans le lieu lisse de $X$. Pour $X^\circ$ on prend alors un ouvert appropri\'e -- c'est-\`a-dire lisse, surjectif et g\'en\'eriquement projectif \`a fibres connexes sur la droite affine $\mathbb{A}^1_k$ -- qui contient l'image de $\mathbb{A}^1_L$ par cette section.

On remarque qu'il suffit (dans le raisonnement ci-dessus) de prendre une extension $L$ qui \emph{quasi}-d\'eploie le tore $T$, comme tout torseur sous un tore quasi-trivial est trivial.

\begin{lemm} \label{brconstant} Soit $K$ une extension finie de $k$ qui quasi-d\'eploie le tore $T$, et telle que le morphisme $p: X^\circ \to \mathbb{A}^1_k$ admet une section $\sigma$ sur $K$. Alors on a $$\mathrm{Br}\,K \stackrel{\sim}{\longrightarrow}\mathrm{Br}\,X^\circ_K.$$\end{lemm}

\begin{proof}[D\'emonstration] La fibre g\'en\'erique $X_\eta^\circ$ du morphisme $X^\circ \to \mathbb{A}^1_k$ est une $\eta$-vari\'et\'e projective, lisse et g\'eom\'etriquement int\`egre. La vari\'et\'e $X_{\eta \otimes_k K}$ est rationnelle sur $\eta \otimes_k K$ -- l'extension $K$ de $k$ quasi-d\'eploie le tore $T$ et tout torseur sous un tore quasi-trivial est trivial. Il en r\'esulte donc que $\mathrm{Br}\,X_{\eta \otimes_k K} = \mathrm{Br}(\eta \otimes_k K)$, car le groupe de Brauer est un invariant birationnel des vari\'et\'es projectives lisses (sur un corps de caract\'eristique z\'ero). D'apr\`es Grothendieck, on dispose donc d'une injection $\mathrm{Br}\, X^\circ_K \hookrightarrow \mathrm{Br}(X_{\eta \otimes_k K})$ qui s'ins\`ere dans le diagramme commutatif suivant:\begin{center} \begin{tikzpicture}
 \matrix (m) [matrix of math nodes, row sep=2em,
    column sep=3em]{
    \mathrm{Br}(X_K^\circ) & \mathrm{Br}(X^\circ_{\eta \otimes_k K})  \\
    \mathrm{Br}(\mathbf{A}^1_K) & \mathrm{Br}(\eta \otimes_k K)\\};
  \path[-stealth]
    (m-1-1) edge node [above] {} (m-1-2) 
    (m-2-2) edge node  [right]  {$\cong$} (m-1-2)
    (m-2-1) edge [bend left] node [left] {$p^\star$} (m-1-1) 
    (m-1-1) edge [bend left]  node [right] {$\sigma^\star$} (m-2-1) 
    (m-2-1) edge node [below] {}(m-2-2) ; 
  \end{tikzpicture} \end{center} 
  
  Il est \'evident que $\sigma^\star \circ p^\star = \mathrm{Id}$; en particulier, l'homomorphisme $\sigma^\star$ est surjectif. La commutativit\'e du diagramme implique que la fl\`eche compos\'ee $$\mathrm{Br}(X_K^\circ) \stackrel{\sigma^\star}{\longrightarrow} \mathrm{Br}(\mathbb{A}^1_K) \longrightarrow \mathrm{Br}(\eta \otimes_k K)$$ est injective. Il en r\'esulte que $\sigma^\star$ est injectif, donc un isomorphisme. 
  
  Comme $\mathrm{Br}(\mathbb{A}^1_K) = \mathrm{Br}(K)$, ceci permet de conclure. \end{proof}

 

\section{Preuve du r\'esultat principal}

\subsection{La preuve} Ce paragraphe est consacr\'e \`a la preuve du Th\'eor\`eme \ref{principal}. On utilise les notations introduites dans le paragraphe pr\'ec\'edent et l'on suppose d\'esormais que $k$ est un corps de nombres. Choisissons $$\alpha_1,\cdots,\alpha_m \in \mathrm{Br}(X^\circ)\ \text{ et }\ \beta_1,\cdots,\beta_n \in \mathrm{Ker}[\mathrm{Br}(X_\eta) \to \mathrm{Br}(X_\eta \otimes_k k^\mathrm{ab})]$$ tels que leurs images engendrent $\mathrm{Br}(X_\eta)/\mathrm{Br}(\eta)$. Ce dernier quotient est un groupe ab\'elien fini, la raison \'etant que la $\eta$-vari\'et\'e $X_\eta$ et projective, lisse, g\'eom\'etriquement connexe et g\'eom\'etriquement rationnelle. Fixons une extension finie ab\'elienne $M$ de $k$ qui quasi-d\'eploie le tore $T$, pour laquelle le morphisme $X^\circ \to \mathbb{A}^1_k$ admet une section apr\`es extension des scalaires \`a $M$ et telle que $$\mathrm{Res}_{k/M}(\beta_i) = 0\ \text{ pour }\ 1 \leq i \leq n.$$

On suppose qu'il n'y ait pas d'obstruction de Brauer-Manin au principe de Hasse et \`a l'approximation faible, c'est-\`a-dire que l'\'egalit\'e $$\sum_{v \in \Omega_k} \mathrm{inv}_v(\mathcal{A}(P_v)) = 0$$ vaut pour tout $\alpha \in \mathrm{Br}(X)$ et tout point ad\'elique $(P_v)_{v \in \Omega_k} \in X(\mathbf{A}_k)$. On fixe un tel point $(P_v)_{v \in \Omega_k} \in X(\mathbf{A}_k)$ et un ensemble fini $S$ de places de $k$. On d\'emontrera que l'on peut approximer les points locaux $(P_v)_{v \in S}$ par un point global $P \in X(k)$. 

En cours de route, on peut toujours agrandir l'ensemble fini de places $S$. On peut donc supposer sans perte de g\'en\'eralit\'e que $S$ contient toutes les places archim\'ediennes de $k$. Consid\'erons le sous-groupe $B = \langle \alpha_1,\cdots,\alpha_m,\beta_1,\cdots,\beta_n \rangle$ de  $\mathrm{Br}(X_\eta)$. Un argument de passage \`a la limite en cohomologie \'etale permet de supposer  $B \subseteq \mathrm{Br}(X_U)$, en r\'etr\'ecissant l'ouvert $U$ si n\'ecessaire. En agrandissant l'ensemble $S$ si n\'ecessaire, on peut supposer que les morphismes $X_U \to U \to k$ admettent des mod\`eles entiers $\mathcal{X}_\mathcal{U} \to \mathcal{U} \to \mathcal{O}_{k,S}$ tels que les fibres de $\mathcal{X}_\mathcal{U} \to \mathcal{U}$ soient toutes projectives, lisses et  g\'eom\'etriquement int\`egres, de fa\c{c}on que $X^\circ \to \mathbb{A}^1_k$ s'\'etende en un morphisme $\mathcal{X}^\circ \to \mathbb{A}^1_{\mathcal{O}_{k,S}}$.  En agrandissant encore $S$, on peut alors supposer que les \'el\'ements de $B$ appartiennent tous au sous-groupe $\text{Br}(\mathcal{X}_\mathcal{U})$ de $\text{Br}(X_\eta)$. Le lemme de Hensel et les estimations de Lang-Weil permettent de supposer \'egalement que pour tout $v \not\in S$ et $\mathcal{Q}$ point ferm\'e de $\mathcal{U}$, la fibre $(\mathcal{X}_\mathcal{U})_\mathcal{Q}$ a un $\kappa(v)$-point ($\kappa(v)$ \'etant le corps r\'esiduel associ\'e \`a $v$).

Comme $X$ est lisse (et g\'eom\'etriquement int\`egre) sur $k$, le th\'eor\`eme des fonctions implicites assure que tout voisinage $v$-adique de $P_v$ est Zariski dense dans $X$. On peut donc remplacer chaque $P_v$ par un $k_v$-point arbitrairement proche, que l'on notera encore $P_v$, qui appartient \`a $X_U$. Notons $Q_v = p(P_v)$ pour tout $v \in S$. 

Choisissons (en utilisant l'approximation faible) $Q_0 \in U(k)$, suffisamment proche de $Q_v$ pour $v \in S$ archim\'edienne, diff\'erent des points $Q_v$. En utilisant un $k$-automorphisme convenable de $\mathbb{P}^1_k$, on peut supposer que $Q_0$ est le point \`a l'infini. Les points $(P_v)_{v \in S}$ appartiennent donc \`a $\mathbb{A}^1_k = \mathrm{Spec}(k[t]) \subseteq \mathbb{P}^1_k$. Notons alors $\lambda_v \in k_v$ pour la $t$-coordonn\'ee de $Q_v \in \mathbb{A}^1(k_v)$. On cherche $Q \in U(k)$ tr\`es proche de $Q_v$ pour toute place finie $v \in S$ et suffisamment grand pour toute place archim\'edienne tel que la fibre $X_Q$ ait un point rationnel, tr\`es proche des points locaux $P_v$ pour $v \in S$.  

Le compl\'ement de $U$ dans $\mathbb{A}^1_k$ est r\'eunion de points $Q_i \in \mathbb{A}^1_k = \mathrm{Spec}(k[t])$, d\'ecrits par des polyn\^omes unitaires  irr\'eductibles $F_i(t) \in k[t]$, pour $1 \leq i \leq s$. Consid\'erons tous les \'el\'ements de $\mathrm{Br}(k(X))$ de la forme $(F_i(t),\chi)$, o\`u $\chi$ est un caract\`ere de $M/k$. Ceux-ci sont d\'efinis par cup-produit $$k(X)^\star \times H^2(k,\mathbf{Z}) \to \mathrm{Br}(k(X)),$$ via l'identification $H^1(k,\mathbf{Q}/\mathbf{Z}) \cong H^2(k,\mathbf{Z})$. Il est clair que ces \'el\'ements s'\'etendent en des alg\`ebres d'Azumaya sur $X_U$. Soit $B'$ le sous-groupe fini de $\mathrm{Br}(X_U)$ engendr\'e par ces \'el\'ements et ceux de $B$. En agrandissant $S$, on peut supposer que les \'el\'ements $(F_i(t),\chi)$ appartiennent \`a $\text{Br}(\mathcal{X}_\mathcal{U})$, i.e. $B' \subseteq \mathrm{Br}(\mathcal{X}_\mathcal{U})$.

Agrandissons encore l'ensemble $S$: on ajoute les places associ\'ees aux polyn\^omes $F_i(t)$ via l'hypoth\`ese (H$_1$) \cite{CTSD94} et les places qui ramifient dans l'extension $M/k$. Consid\'erons l'image de $\mathrm{Br}(X^\circ)$ dans $\mathrm{Br}(X_M^\circ)$. Le lemme \ref{brconstant} implique que pour tout $i$, on peut trouver un \'el\'ement $\gamma_i \in \mathrm{Br}\,M$ tel que $\gamma_i$ s'envoie sur l'image de $\alpha_i$ via la fl\`eche canonique $\mathrm{Br}\,M \to \mathrm{Br}(X_M^\circ)$. En prenant $S$ suffisamment grand, on peut maintenant supposer que $\mathrm{inv}_w\ \gamma_i = 0$ pour $i \in \{1,\cdots,m\}$ et pour $w \not\in S_M$ ($S_M$ \'etant l'ensemble des places de $M$ au-dessus des places dans $S$).

On suppose qu'il n'y a pas d'obstruction de Brauer-Manin \`a l'approximation faible pour $(P_v)_{v \in S}$. Le lemme formel de Harari (\cite[Corollaire 2.6.1]{Har}) donne alors l'existence d'un ensemble fini de places $S_1 \supseteq S$ et de points locaux $P_v \in X_U(k_v)$ pour $v \in S_1 \setminus S$, tels que pour tout $\mathcal{A} \in B'$, $$\sum_{v \in S_1} \text{inv}_v (\mathcal{A}(P_v)) = 0 \in \mathbf{Q}/\mathbf{Z}.$$ L'hypoth\`ese (H$_1$) donne alors l'existence de $\lambda \in k$ qui est suffisamment proche de $\lambda_v$ pour toute place finie $v \in S_1$, entier en dehors de $S_1$ et suffisamment grand pour les places archim\'ediennes, tel qu'on ait (simultan\'ement):

\begin{itemize}
\item $\lambda \in U(k)$;
\item pour $v \in S_1$, la fibre $X_\lambda$ contient un $k_v$-point $P_v'$ qui est $v$-adiquement proche de $P_v$, tel que $\text{inv}_v(\mathcal{A}(P_v)) = \text{inv}_v(\mathcal{A}(P_v'))$ pour tout $\mathcal{A} \in B'$;
\item pour $1 \leq i \leq s$, il existe une place $v_i$ de $k$ telle que $F_i(\lambda)$ est une uniformisante dans $\mathcal{O}_{k,v_i}$ et une unit\'e dans $k_v$ pour toute $v \not\in S_1 \cup \{v_i\}$. \end{itemize}

Pour $v \not\in S_1$ et $v \neq v_i$ pour tout $i$, on obtient alors, gr\^ace au lemme de Hensel, un $\mathcal{O}_v$-point de la fibre $(\mathcal{X}_\mathcal{U})_{\lambda}$. L'ensemble $S$ a \'et\'e pris suffisamment grand pour que $B' \subseteq \text{Br}(\mathcal{X}_\mathcal{U})$. On obtient donc de cette fa\c{c}on un $\mathcal{O}_v$-point $P_v'$ de $X_{\lambda}$ tel que $\text{inv}_v\,\alpha(P_v') = 0$ pour tout $\alpha \in B'$ (``bonne r\'eduction'').

Analysons maintenant la situation si $v = v_i$ est l'une des ``places de Schinzel''. Notons que par la loi de r\'eciprocit\'e globale, on a, pour tout caract\`ere $\chi$ de $M/k$, $$\sum_{v \in \Omega_k} \text{inv}_v (F_i(\lambda),\chi) = 0.$$ Par construction, on a \'egalement  $$\sum_{v \in S_1} \text{inv}_v\, (F_i(\lambda),\chi) = 0.$$ Pour $v \not\in S_1$ et $v \neq v_i$, on sait que $F_i(\lambda)$ est une unit\'e dans $k_{v}$. On trouve donc $$\text{inv}_{v_i}\,(F_i(\lambda),\chi) = 0.$$ Cette  derni\`ere \'egalit\'e vaut pour tout caract\`ere $\chi$ de $M/k$. Il en r\'esulte que la place $v_i$ est totalement d\'ecompos\'ee dans l'extension (ab\'elienne!) $M/k$, car (par construction) $F_i(\lambda)$ est une uniformisante dans $\mathcal{O}_{k,v_i}$.

La fibre $X_\lambda$ est une compactification lisse d'un espace principal homog\`ene sous $T$, qui devient quasi-trivial apr\`es extension des scalaires \`a $M/k$. Il en r\'esulte que le $k_{v_i}$-tore $T_{k_{v_i}}$ est quasi-trivial, donc $H^1(k_{v_i},T_{k_{v_i}}) = 0$. Ceci implique la $k_{v_i}$-rationalit\'e de $X_\lambda \otimes_k k_{v_i}$. Il existe donc un $k_{v_i}$-point $P_{v_i}'$.

Montrons que $\text{inv}_{v_i}\,\mathcal{A}(P_{v_i}') = 0$ pour $\mathcal{A} \in B$. Comme $v_i$ est totalement d\'ecompos\'ee dans $M/k$, la fl\`eche $\langle \alpha_1,\cdots,\alpha_m \rangle \to \mathrm{Br}(X^\circ_{k_{v_i}})$ se factorise en $$\langle \alpha_1,\cdots,\alpha_m \rangle \to \text{Br}(X_M^\circ) \to \text{Br}(X^\circ_{k_{v_i}}).$$ Pour $1 \leq j \leq m$, on trouve alors $$\alpha_j(P_{v_i}') = \left.\alpha_j\right|_{k_{v_i}}(P_{v_i}') = \left.\gamma_j\right|_{k_{v_i}}(P_{v_i}') =  0.$$ Pour $1 \leq \ell \leq n$, on obtient $\beta_\ell(P_{v_i}') = \left.\beta_\ell\right|_{k_{v_i}}(P_{v_i}')=0$ car $\left.\beta_\ell\right|_M = 0$. Donc $$\text{inv}_{v_i}\,\mathcal{A}(P_{v_i}') = 0$$ pour $\mathcal{A} \in B$.

Faisons un r\'esum\'e de la situation. On a construit pour toute $v \in \Omega_k$ un $k_v$-point $P_v'$ de la fibre $X_\lambda$. On a $\sum_{v \in S_1} \text{inv}_v\,\mathcal{A}(P_v') = 0$ et $\text{inv}_v\,\mathcal{A}(P_v') = 0$ pour tout $\mathcal{A} \in B$ et $v \not\in S_1$. On obtient donc $\sum_{v \in \Omega_k} \text{inv}_v\,\mathcal{A}(P_v') = 0$, i.e. on a construit un point ad\'elique de la fibre $X_\lambda$ orthogonal \`a $B$. Or, le Lemme \ref{spec} assure que la fl\`eche $$\mathrm{Br}(X_\eta)/\mathrm{Br}(\eta) \to \mathrm{Br}(X_\lambda)/\mathrm{Br}(k)$$ est un isomorphisme. L'image du groupe $B$ dans le quotient $\mathrm{Br}(X_\eta)/\mathrm{Br}(\eta)$ engendre celui-ci. Le point ad\'elique $(P_v)_{v \in \Omega_k}$ est donc orthogonal au groupe de Brauer de $X_\lambda$, qui est une compactification lisse d'un espace principal homog\`ene sous un $k$-tore. Un r\'esultat classique de Sansuc \cite[Corollaires 8.7 et 8.13]{Sansuc} implique que l'on peut approximer ce point ad\'elique par un point global $P \in X_\lambda(k)$, ce qui ach\`eve la D\'emonstration $\square$

\subsection{Sur la remarque \ref{addcomb}} Expliquons pourquoi le r\'esultats r\'ecents de Green, Tao et Ziegler en combinatoire additive impliquent une version inconditionnelle du Th\'eor\`eme \ref{principal} si $k = \mathbf{Q}$ et si toutes les fibres non-scind\'ees sont d\'efinies sur $\mathbf{Q}$. 

Dans la preuve ci-dessus, on peut alors supposer que l'ouvert $U$ est le compl\'ement d'un nombre fini de points rationnels. Les polyn\^omes $F_i(t)$ dans la preuve sont donc tous de la forme $t - e_i$, pour $e_1,\cdots,e_s \in \mathbf{Q}$. L\`a o\`u on a choisi $\lambda \in \mathbf{Q}$ utilisant l'hypoth\`ese (H$_1$), on peut maintenant choisir $\lambda$ utilisant \cite[Proposition 2.1]{HSW}. On trouve donc $\lambda \in \mathbf{Q}$ arbitrairement proche de $\lambda_v$ pour $v \in S_1$ finie et suffisamment grand pour la topologie r\'eelle, tel que $\lambda$ est une uniformisante dans $\mathbf{Q}_{v_i}$ pour $1 \leq i \leq s$ et $$\mathrm{val}_{v}(\lambda - e_i) \leq 0\ \text{ pour }\ v \not\in S_1 \cup \{v_1,\cdots,v_s\}.$$ Cette derni\`ere in\'egalit\'e implique que la r\'eduction de $\lambda$ modulo $v$ est un \'el\'ement de $\mathbb{P}^1(\kappa(v))$ autre que les r\'eductions de $e_1,\cdots,e_s$. Les estimations de Lang-Weil et le lemme de Hensel impliquent que ce $\kappa(v)$-point se rel\`eve en un $\mathcal{O}_v$-point sur un mod\`ele entier. Le reste de l'argument ne change pas.

\section{Applications}

 Comme on l'a d\'ej\`a expliqu\'e dans l'introduction, notre th\'eor\`eme principal g\'en\'eralise \cite[Theorem 3.1]{Wei11}. Il englobe \'egalement \cite[Theorem 3.2]{Wei11}, qui concerne les mod\`eles projectifs et lisses des vari\'et\'e affines d\'efinies par une \'equation de la forme $$(X_1^2 - aY_1^2)(X_2^2 - bY_2^2)(X_3^2 - abY_3^2) = P(t),$$ pour $a,b \in k$. Une telle vari\'et\'e est fibr\'ee sur $\mathbb{P}^1_k$ via la projection sur $t$. Les fibres lisses sont des torseurs sous le tore noyau du morphisme de tores \begin{eqnarray*} R_{k(\sqrt{a})/k}\mathbb{G}_m \times R_{k(\sqrt{b})/k}\mathbb{G}_m \times  R_{k(\sqrt{ab})/k}\mathbb{G}_m & \longrightarrow &  \mathbb{G}_m \\  (\mathbf{x},\mathbf{y},\mathbf{z}) & \mapsto &  N_{k(\sqrt{a})/k}(\mathbf{x}) \cdot N_{k(\sqrt{b})/k}(\mathbf{y}) \cdot N_{k(\sqrt{ab})/k}(\mathbf{z}). \end{eqnarray*} Colliot-Th\'el\`ene a d\'emontr\'e \cite[Th\'eor\`eme 4.1]{Colliot} que si aucune des constantes $a$, $b$ et $ab$ n'est un carr\'e dans $k$, alors  $(X_1^2 - aY_1^2,b)$ est l'unique \'el\'ement non-trivial de $\mathrm{Br}(X_\eta)/\mathrm{Br}(\eta)$. Comme cet \'el\'ement de $\mathrm{Br}(X_\eta)$ s'annule apr\`es extension des scalaires \`a l'extension ab\'elienne $k(\sqrt{b})/k$, le th\'eor\`eme \ref{principal} s'applique.

Dans son texte \cite{Liang12}, Liang a obtenu des r\'esultats sur les fibrations sous une hypoth\`ese technique (Br) sur le  groupe de Brauer (voir par exemple \cite[Theorem 2.1]{Liang12}) qui para\^it moins restrictive que la n\^otre. On montre maintenant que ces hypoth\`eses sont en fait \'equivalentes:

\begin{prop} \label{liang} Avec les notations utilis\'ees pr\'ec\'edemment au \S 2, on consid\`ere des points ferm\'es $Q_1, Q_2, \cdots,Q_r$ de $\mathbb{A}^1_k$ tels que les fibres $Y_1,\cdots,Y_r$ au-dessus de ces points sont lisses et g\'eom\'etriquement int\`egres. Si $\mathrm{Br}\left(X^\circ \setminus \bigsqcup_{i = 1}^r Y_i\right)$ engendre $\mathrm{Br}(X_\eta)/\mathrm{Br}(\eta)$, alors $\mathrm{Br}(X^\circ)$ engendre $\mathrm{Br}(X_\eta)/\mathrm{Br}(\eta)$. \end{prop}

\begin{proof}[D\'emonstration] Prenons une classe $\overline{\alpha} \in \mathrm{Br}\,X_\eta/\mathrm{Br}\,\eta$ et $\alpha \in \mathrm{Br}\left(X^\circ \setminus \bigsqcup_{i = 1}^r Y_i\right)$ qui s'envoie sur $\overline{\alpha}$. Supposons que $\alpha$ ramifie au point g\'en\'erique de $Y_i$. Soit $$\partial_i: \mathrm{Br}\,(X_\eta) \to H^1(Y_i,\mathbf{Q}/\mathbf{Z})$$ l'application r\'esidu. On voit facilement que $$\mathrm{Br}(X_{\eta \otimes_k \overline{k}}) \cong \mathrm{Br}(\eta \otimes_k \overline{k}) = 0$$ (th\'eor\`eme de Tsen), car $X_{\eta \otimes_k \overline{k}}$ est projective, lisse et $(\eta \otimes_k \overline{k})$-rationnelle. 

On a donc le diagramme commutatif suivant: \begin{center}
\begin{tikzpicture}[auto]
\node (L1) {$\mathrm{Br}(X_\eta)$};
\node (L2) [below= 1cm of L1] {$\mathrm{Br}(X_{\eta \otimes_k \overline{k}})$};
\node (M1) [right=1.5cm  of L1] {$H^1(Y_i,\mathbf{Q}/\mathbf{Z})$};
\node (M2) at (M1 |- L2) {$H^1(\overline{Y_i},\mathbf{Q}/\mathbf{Z})$};
\draw[->] (L1) to node {\footnotesize $\partial_i$} (M1);
\draw[->] (L2) to node {\footnotesize $\partial_i$} (M2);
\draw[->] (L1) to node {} (L2);
\draw[->] (M1) to node {} (M2);
\end{tikzpicture}
 \end{center} Ceci montre que $\partial_i(\alpha)$ appartient au noyau de l'application $$H^1(Y_i,\mathbf{Q}/\mathbf{Z}) \to H^1(\overline{Y}_i,\mathbf{Q}/\mathbf{Z}).$$ Si $K_i = k(Q_i)$ il en r\'esulte que $\partial_i(\alpha)$ provient d'un caract\`ere $\chi_i \in H^1(K_i,\mathbf{Q}/\mathbf{Z})$ -- c'est l\`a que l'on utilise le fait que $Y_i$ est g\'eom\'etriquement int\`egre.  Consid\'erons $$\beta_i := \mathrm{Cores}_{K_i/k}(\chi_i, t - e_i) \in \mathrm{Br}\,\eta,$$ o\`u $e_i$ est la classe de $t$ dans le quotient $K_i = k[t]/(P_i(t))$, o\`u $P_i(t)$ est le polyn\^ome qui d\'efinit le point ferm\'e $Q_i$. Regardons le diagramme commutatif \begin{center}
\begin{tikzpicture}[auto]
\node (L1) {$\mathrm{Br}(\eta \otimes_k K_i)$};
\node (L2) [below= 1.2cm of L1] {$\eta$};
\node (M1) [right=1.8cm  of L1] {$\bigoplus_{M \mapsto Q_i} H^1(K_i(M),\mathbf{Q}/\mathbf{Z})$};
\node (M2) at (M1 |- L2) {$H^1(k(Q_i),\mathbf{Q}/\mathbf{Z})$};
\draw[->] (L1) to node {} (M1);
\draw[->] (L2) to node {} (M2);
\draw[->] (L1) to node {} (L2);
\draw[->] (M1) to node {} (M2);
\end{tikzpicture}
 \end{center}
dans lequel les fl\`eches horizontales sont les applications r\'esidu et les fl\`eches verticales sont les applications de corestriction. Ceci montre que $\partial_{Q_i}(\beta_i) = \chi_i$. L'alg\`ebre $\beta_i$ est clairement non-ramifi\'ee en dehors de $Q_i$. On voit maintenant que $\alpha - \sum_{i = 1}^r \beta_i$ appartient non seulement \`a $\mathrm{Br}(X^\circ \setminus \bigsqcup Y_i)$ mais aussi au sous-groupe $\mathrm{Br}\,X^\circ$. Cet \'el\'ement s'envoie toujours sur $\overline{\alpha} \in \mathrm{Br}\,X_\eta/\mathrm{Br}\,\eta$. Il en r\'esulte que $\mathrm{Br}\,X^\circ \longrightarrow \mathrm{Br}\,X_\eta/\mathrm{Br}\,\eta$ est surjective.  \end{proof} 

Voici un exemple d'application du Th\'eor\`eme \ref{principal} (cfr. Corollaire \ref{coroll}):

\begin{prop} \label{cor} Soit $T$ un $k$-tore qui est quasi-d\'eploy\'e par une extension cyclique. Soit $f: X \to \mathbb{P}^1_k$ une fibration en torseurs sous $T$. Admettons l'hypoth\`ese de Schinzel. Alors $X(k)$ est dense dans $X(\mathbf{A}_k)^{\mathrm{Br}(X)}$. \end{prop}

\begin{proof}[D\'emonstration] Ceci est une cons\'equence imm\'ediate du Th\'eor\`eme \ref{principal} si on v\'erifie que $\mathrm{Ker}[\mathrm{Br}(X_\eta) \to \mathrm{Br}(X_\eta \otimes_k L)]$ engendre $\mathrm{Br}(X_\eta)/\mathrm{Br}(\eta)$, o\`u $L/k$ est l'extension cyclique qui quasi-d\'eploie le $k$-tore $T$. Pour l'ouvert $X^\circ$ dans le Th\'eor\`eme \ref{principal}, on peut alors prendre n'importe quel ouvert de $X$ qui est lisse, surjectif, g\'en\'eriquement projectif \`a fibres connexes sur $\mathbb{A}^1_k$, tel que le morphisme induit $X^\circ \to \mathbb{A}^1_k$ admette une section sur $L$ (cfr. la discussion qui pr\'ec\`ede le Lemme \ref{brconstant}).

Or, ceci r\'esulte facilement de \cite[Proposition 1.1.(b)]{Colliot}, appliqu\'ee \`a $X_\eta$. La $\eta$-vari\'et\'e $X_\eta$ devient bien rationnelle apr\`es extension des scalaires \`a $L$, car le tore $T \otimes_k L$ est quasi-trivial, et tout torseur sous un tore quasi-trivial est trivial. \end{proof} 

  Ceci donne une g\'en\'eralisation de \cite[Theorem 3.2]{Wei11}. Comme expliqu\'e ci-dessus, on obtient le m\^eme r\'esultat sans admettre l'hypoth\`ese de Schinzel si $k = \mathbf{Q}$ et toutes les fibres non-scind\'ees sont d\'efinies sur $\mathbf{Q}$.


\section{Z\'ero-cycles}

On \'etablit maintenant l'analogue naturel du Th\'eor\`eme \ref{principal} pour les z\'ero-cycles de degr\'e $1$ au lieu des points rationnels. L'astuce de Salberger servira comme substitut pour l'hypoth\`ese de Schinzel dans notre preuve. On renvoie \`a \cite[Section 3.1]{CTSD94} pour la d\'efinition et les propri\'et\'es de base de l'accouplement entre le groupe de Brauer d'une vari\'et\'e et les z\'ero-cycles sur cette vari\'et\'e.

\begin{theo} \label{0cyc} Soit $T$ un $k$-tore quasi-d\'eploy\'e par $k^{\mathrm{ab}}$. Soit $f: X \to \mathbb{P}^1_k$ une fibration en torseurs sous $T$. Supposons qu'il existe un ouvert $X^\circ \subseteq X$ qui est lisse, surjectif et g\'en\'eriquement projectif \`a fibres connexes au-dessus de $\mathbb{A}^1_k$, de fa\c{c}on que le morphisme $X^\circ \to \mathbb{A}^1_k$ admette une section sur $k^{\mathrm{ab}}$ et que $\mathrm{Br}(X_\eta)/\mathrm{Br}(\eta)$ soit engendr\'e par les images des groupes suivants:
\begin{itemize}
\item le groupe $\mathrm{Br}(X^\circ)$, et
\item le groupe $\mathrm{Ker}[\mathrm{Br}(X_\eta) \to \mathrm{Br}(X_\eta \otimes_k k^\mathrm{ab})]$. \end{itemize}
Alors  l'obstruction de Brauer-Manin \`a l'existence d'un z\'ero-cycle de degr\'e $1$ sur $X$ est la seule. \end{theo}

\begin{proof}[D\'emonstration] 
On utilise les notations du \S 2. Choisissons $$\alpha_1,\cdots,\alpha_m \in \mathrm{Br}(X^\circ)\ \text{ et }\ \beta_1,\cdots,\beta_n \in \mathrm{Ker}[\mathrm{Br}(X_\eta) \to \mathrm{Br}(X_\eta \otimes_k k^\mathrm{ab})]$$ et une extension ab\'elienne $M$ de $k$ comme dans la preuve du Th\'eor\`eme \ref{principal}. Soit $B$ le sous-groupe $\langle \alpha_1,\cdots,\alpha_m,\beta_1,\cdots,\beta_n \rangle$ de  $\mathrm{Br}(X_\eta)$. En r\'etr\'ecissant l'ouvert $U$ si n\'ecessaire, on peut supposer sans perte de g\'en\'eralit\'e que $B \subseteq \mathrm{Br}(X_U)$. Le compl\'ement de $U$ dans $\mathbb{A}^1_k$ est r\'eunion de points ferm\'es $Q_i \in \mathbb{A}^1_k = \mathrm{Spec}(k[t])$, d\'ecrits par des polyn\^omes unitaires  irr\'eductibles $F_i(t) \in k[t]$, pour $1 \leq i \leq s$.

Fixons d'abord un ensemble fini de places $S_0 \subseteq \Omega_k$ tel que les conditions suivantes soient satisfaites. Les morphismes $X_U \to U \to k$ admettent des mod\`eles $\mathcal{X}_\mathcal{U} \to \mathcal{U} \to \mathcal{O}_{k,S}$ tels que les fibres de $\mathcal{X}_\mathcal{U} \to \mathcal{U}$ soient projectives, lisses et  g\'eom\'etriquement int\`egres, de fa\c{c}on que $X^\circ \to \mathbb{A}^1_k$ s'\'etende en un morphisme $\mathcal{X}^\circ \to \mathbb{A}^1_{\mathcal{O}_{k,S_0}}$. On a des inclusions $$B \subseteq \text{Br}(\mathcal{X}_\mathcal{U}) \subseteq \text{Br}(X_\eta).$$ Pour $v \not\in S_0$ et $\mathcal{Q}$ point ferm\'e de $\mathcal{U}$, la fibre $(\mathcal{X}_\mathcal{U})_\mathcal{Q}$ a un $\kappa(v)$-point. Pour $v \not\in S_0$, l'extension $M/k$ est non-ramifi\'ee en $v$. Finalement, pour $\mathcal{A} \in B \cap \mathrm{Br}(X^\circ)$, pour $w$ une place de $M$ non au-dessus d'une place de $S_0$ et $$j: \mathrm{Br}(X^\circ) \to \mathrm{Br}(X^\circ_M) = \mathrm{Br}_0(X^\circ_M)$$ la fl\`eche naturelle, on a $\mathrm{inv}_w(j(\mathcal{A})) = 0$.




Soit $B'$ le sous-groupe de $\mathrm{Br}(\mathcal{X}_\mathcal{U})$ engendr\'e par les \'el\'ements de $B$ et ceux de la forme $(F_i(t),\chi)$, pour $\chi$ un caract\`ere de l'extension $M/k$. Une variante du lemme formel de Harari (voir \cite[Theorem 3.2.2]{CTSD94}) permet alors de supposer que l'on dispose d'un ensemble fini $S_1 \subseteq \Omega_k$ contenant $S_0$ et pour toute $v \in S_1$, un z\'ero-cycle $z_v$ de degr\'e $1$, support\'e sur $X_U \otimes_k k_v$, tel que $$\sum_{v \in S_1} \text{inv}_v\left( \langle \mathcal{A}, z_v \rangle \right) = 0$$ pour tout $\mathcal{A} \in B'$. Notons $e$ l'exposant du groupe fini $B'$. Pour toute place $v$, on a une \'ecriture (\'evidente) $z_v = z_v^+ - z_v^-$, o\`u $z_v^+$ et $z_v^-$ sont des z\'ero-cycles effectifs. \'Ecrivons $$\tilde{z}_v = z_v + nez_v^-,\text{ o\`u } n = [M : k].$$ On obtient $$\sum_{v \in S_1} \text{inv}_v\left(\langle \mathcal{A}, \tilde{z}_v \rangle \right) = 0$$ pour tout $\mathcal{A} \in B'$. Choisissons un z\'ero-cycle effectif $\zeta$ sur $X_U$ qui est de degr\'e $n$. Alors $\langle \mathcal{A}, eP \rangle = 0$ pour $\mathcal{A} \in B'$. Comme $\deg \tilde{z}_v \equiv 1 \mod{ne}$ et $\deg e\zeta = ne$, on obtient -- en ajoutant des multiples appropri\'es du z\'ero-cycle $e\zeta$ \`a chacun des $\tilde{z}_v$ -- des cycles effectifs $\overline{z}_v$, tous du m\^eme degr\'e $D \equiv 1 \pmod{ne}$, tels que $$\sum_{v \in S_1} \text{inv}_v\left( \langle \mathcal{A}, \overline{z}_v \rangle \right) = 0$$ pour tout $\mathcal{A} \in B'$. En utilisant \cite[Lemma 6.2.1]{CTSD94}, on voit que l'on peut supposer que chacun de ces z\'ero-cycles effectifs $\overline{z}_v$ est somme de points ferm\'es (sans multiplicit\'es), dont les images sous $f: X_U \to U$ sont toutes diff\'erentes. On peut \'egalement supposer que pour tout point ferm\'e $Q$ dans le support du z\'ero-cycle $\overline{z}_v$, l'extension de corps $k_v(f(P)) \subseteq k_v(P)$ est un isomorphisme: on remplace chaque point ferm\'e $Q$ par un point $k_v(Q)$-rationnel de $X_U \otimes_k k_v$ suffisamment proche de $Q$.

 Consid\'erons maintenant les z\'ero-cycles (effectifs) $f(\overline{z}_v)$. Chacun de ces cycles est donn\'e par un polyn\^ome unitaire et s\'eparable $G_v(t) \in k_v[t]$ de degr\'e $D$, premier au produit $\prod_{i = 1}^s F_i(t)$. Les fibres de $f$ au-dessus des racines de $G_v(t)$ admettent un point rationnel, pour $v \in S_1$. Il en est de m\^eme pour tout polyn\^ome unitaire dans $k_v[t]$ suffisamment proche de $G_v(t)$ pour les topologies $v$-adiques sur l'espace des coefficients (Krasner).

 Choisissons maintenant un polyn\^ome unitaire, irr\'eductible $G(t) \in k[t]$ de degr\'e $D$ comme dans \cite[Theorem 3.1]{CTSkoSD98} (``l'astuce de Salberger''); pour le corps $L$ dans ce th\'eor\`eme, on prend une extension finie de $k$ qui contient $M$ et sur laquelle les $F_i(t)$ sont scind\'ees. Pour $V$, on prend l'ensemble $T$ des places de $k$ qui scindent les $F_i(t)$ et qui sont totalement d\'ecompos\'ees dans $L$. On peut choisir $G(t)$ suffisamment proche des $G_v(t)$ pour toute place $v \in S_1$ pour que les conditions suivantes soient satisfaites:  \begin{itemize}
 
 \item si $$N = k[t]/(G(t)),\ \theta = \overline{t} \in N,$$ alors pour chacun des polyn\^omes $F_i(t)$ (pour $1 \leq i \leq s$), il existe une place $w_i$ de $N$, telle que $F_i(\theta)$  soit une uniformisante dans la compl\'etion $N_{w_i}$ et une unit\'e dans $N_w$ si $w$ n'est pas au-dessus de $S_1 \cup T$ (et $w \neq w_i$); \item pour $v \in S_1$ et toute place $w$ au-dessus de $v$, la fibre $X_\theta$ (d\'efinie sur $N$) admet un $N_w$-point $Q_w$, avec la somme $\sum_{w \mid v} Q_w$ suffisamment proche du z\'ero-cycle $\overline{z}_v$ pour que l'\'egalit\'e $$\sum_{w \mid v} \text{inv}_v(\text{Cor}_{N_w/k_v}(\mathcal{A}(Q_w)) = \text{inv}_v\left(\langle \mathcal{A}, \overline{z}_v \rangle \right)$$ vaille pour tout $\mathcal{A} \in B'$. \end{itemize}
 
Montrons maintenant que $X_\theta$ admet un $N_w$-point $Q_w$ pour toute place $w$ de $N$ qui n'est pas au-dessus de $S_1$. Si $w$ n'est pas non plus au-dessus de $T$ et n'est pas \'egale \`a l'une des $w_i$ pour $1 \leq i \leq s$, il est clair que $(\mathcal{X}_\mathcal{U})_\theta$ admet un $\mathcal{O}_{N_w}$-point. 

Si la place $w$ est au-dessus de $T$, alors $w$ est totalement d\'ecompos\'ee dans l'extension $MN$ de $N$; il est donc \'evident que $X_\theta$ admet un $N_w$-point. Concentrons-nous donc sur le cas $w = w_i$ pour un $i \in \{1,\cdots,s\}$. On sait que $$\sum_{v \in S_1, w \mid v} \text{inv}_v \left(\text{Cor}_{N_w/k_v}(F_i(\theta), \chi)\right) = 0$$ et donc $$\sum_{v \notin S_1, w \mid v} \text{inv}_v \left(\text{Cor}_{N_w/k_v}(F_i(\theta), \chi)\right) = 0.$$ 

Or, si $w$ n'est pas au-dessus de $S_1 \cup T$ et $w \neq w_i$, alors $F_i(\theta)$ est une unit\'e dans $N_w$, et si $w$ est au-dessus de $T$, alors $w$ est totalement d\'ecompos\'ee dans l'extension $MN$ de $N$. Dans les deux cas, $\text{inv}_w(F_i(\theta),\chi) = 0$. Il en r\'esulte que $$\text{inv}_{w_i}(F_i(\theta),\chi) = 0$$ pour tout caract\`ere $\chi$ de $M/k$.  Comme $F_i(\theta) \in N$ est une uniformisante dans $N_{w_i}$ et l'extension $M/k$ est ab\'elienne, on voit que $w_i$ est totalement d\'ecompos\'ee dans l'extension $MN$ de $N$. La fibre $X_\theta$ admet donc un $N_{w_i}$-point.

 Exactement comme dans la preuve du th\'eor\`eme \ref{principal}, on montre $\mathrm{inv}_w(\mathcal{A}(Q_w)) = 0$ pour tout $\mathcal{A} \in B$ et toute $w \in \Omega_N$ non au-dessus de $S_1$. 

On obtient donc l'\'egalit\'e $$\sum_{w \in \Omega_N} \mathrm{inv}_w(\mathcal{A}(Q_w)) = 0.$$ La fl\`eche naturelle $$\mathrm{Br}(T^c)/\mathrm{Br}_0(T^c) \to \mathrm{Br}(T^c_N)/\mathrm{Br}(T^c_N)$$ est un isomorphisme, car $M \cap N = k$. L'image de $B$ engendre donc bien le quotient $\mathrm{Br}(X_\theta)/\mathrm{Br}_0(X_\theta)$. Une application du r\'esultat de Sansuc \cite[Corollaire 8.7]{Sansuc} \`a la fibre $X_\theta$ donne l'existence d'un $N$-point sur la fibre $X_\theta$. Comme $[N : k] \equiv 1 \pmod{ne}$, ceci implique bien l'existence d'un z\'ero-cycle de degr\'e $1$ sur $X$. \end{proof}

\section{Une remarque sur l'hypoth\`ese de Schinzel}

On pr\'esente dans ce dernier paragraphe une remarque sur l'hypoth\`ese de Schinzel qui permet de l'utiliser de fa\c{c}on plus souple. Il s'agit du choix d'un param\`etre de Schinzel $\lambda$, comme dans \cite[hypoth\`ese (H$_1$)]{CTSD94}, avec la condition suppl\'ementaire que cette valeur de $\lambda$ appartienne \`a un sous-ensemble hilbertien donn\'e de $k$. Que ceci devrait \^etre possible semble avoir \'et\'e sugg\'er\'e pour la premi\`ere fois dans \cite{SDLetter}.

\subsection{Sur la topologie d'un ensemble Hilbertien}  Il nous faut un r\'esultat pr\'eliminaire sur la topologie d'un ensemble hilbertien; pour cette derni\`ere notion, voir par exemple \cite{Ekedahl}.

\begin{prop} \label{HilbertSchinzel} Soit $k$ un corps de nombres. Soit $V$ une $k$-vari\'et\'e g\'eom\'etriquement int\`egre. Soit $\rho: X \to V$ un rev\^etement \'etale connexe et soit $P \in V(k)$ tel que la fibre de $\rho$ en $P$ soit connexe. Soit $S \subseteq \Omega_k$ un ensemble fini de places de $k$. Alors il existe un ensemble fini $T \subseteq \Omega_k^{<\infty}$ avec $S \cap T = \emptyset$ et un voisinage $N$ de $P$ dans $\prod_{v \in T} V(k_v)$ dans la topologie naturelle tel que pour tout $Q \in N(k)$, la fibre $\rho^{-1}(Q)$ soit connexe.
\end{prop}

\begin{proof}[D\'emonstration] On agrandissant $S$ si n\'ecessaire, on peut supposer que la situation s'\'etend sur $\mathcal{O}_{k,S}$. Partons donc d'un morphisme dominant $\pi: \mathcal{V} \to \mathrm{Spec}(\mathcal{O}_{k,S})$ dont la fibre g\'en\'erique est une $k$-vari\'et\'e g\'eom\'etriquement int\`egre, d'un rev\^etement \'etale connexe $\rho:\mathcal{X} \to \mathcal{V}$ de fibre g\'en\'erique int\`egre et d'un point $P \in \mathcal{V}(\mathcal{O}_{k,S})$ tel que la fibre de $\rho$ en $P$ soit connexe. On peut supposer \'egalement que $\rho$ est Galoisien de groupe $G$. 

D\'emontrons d'abord que pour tout  \'el\'ement $g \in G$, il existe une place finie $v_g \in \Omega_k^{< \infty} \setminus S$ telle que $\mathrm{Frob}_{v_g}\,P$ est conjugu\'e \`a $g$. Fixons $g \in G$. \'Ecrivons $\mathcal{X}' = \mathcal{X}/\langle g \rangle$, le quotient de $\mathcal{X}$ par l'action de $\langle g \rangle$, qui est un rev\^etement \'etale de $\mathcal{V}$. Consid\'erons le diagramme commutatif suivant, dont les carr\'es sont Cart\'esiens:  \begin{center}
\begin{tikzpicture}[auto]
\node (L1) {$\mathcal{X}$};
\node (L2) [below= 1.2cm of L1] {$\mathcal{X}'$};
\node (L3) [below= 1.2cm of L2] {$\mathcal{V}$};
\node (M1) [right= 1.5cm  of L1] {$\mathrm{Spec}\,B$};
\node (M2) at (M1 |- L2) {$\mathrm{Spec}\,A$};
\node (M3) at (M1 |- L3) {$\mathrm{Spec}\,\mathcal{O}_{k,S}$};
\node (R1) [right= 1.2cm  of M1] {$\mathrm{Spec}\,M$};
\node (R2) at (R1 |- M2) {$\mathrm{Spec}\,L$};
\node (R3) at (R1 |- M3) {$\mathrm{Spec}\,k$};

\draw[<-] (L1) to node {} (M1);
\draw[<-] (M1) to node {} (R1);
\draw[<-] (L2) to node {} (M2);
\draw[<-] (M2) to node {} (R2);
\draw[<-] (L3) to node {\footnotesize $P$} (M3);
\draw[<-] (M3) to node {} (R3);
\draw[->] (L1) to node {\footnotesize $\langle g \rangle$} (L2);
\draw[->] (L2) to node {} (L3);
\draw[->] (M1) to node {} (M2);
\draw[->] (M2) to node {} (M3);
\draw[->] (R1) to node {} (R2);
\draw[->] (R2) to node {} (R3);
\end{tikzpicture}
 \end{center} 
La fibre du morphisme $\rho: \mathcal{X} \to \mathcal{V}$ au-dessus du $k$-point $P$ est connexe. Les $k$-alg\`ebres $L$ et $M$ sont donc bien des corps et l'extension $M/L$ est cyclique, de groupe $\langle g \rangle$. Ceci implique que $A/\mathcal{O}_{k,S}$ est une extension finie d'anneaux de Dedekind. 

Donc \cite[Proposition 2.2.1]{Har} s'applique: il existe une place $v_g$ de $k$ en dehors de $S$ et une place $w$ de $L$ au-dessus de $v_g$ qui est de degr\'e $1$  et qui est inerte dans $M$. Par construction, $\mathrm{Frob}_{v_g}\,P$ est conjugu\'e \`a $g$.  

Ceci \'etant \'etabli, choisissons pour tout \'el\'ement $g \in G$ une telle place $v_g$ et prenons $T = \{v_g \mid g \in G\}$. Prenons pour $\mathcal{N}$ l'ensemble des $k$-points de $\mathcal{V}$ qui sont entiers en les places de $T$ et qui, pour tout $g \in G$, ont la m\^eme r\'eduction (comme $\kappa(v_g)$-point) que $P$. Il suffit de montrer que $\rho^{-1}(Q)$ est connexe pour tout $Q \in \mathcal{N}$. Ceci \'equivaut \`a dire que le groupe de Galois $\mathrm{Gal}(\overline{k}/k)$ se surjecte sur $G$ \`a travers un rel\`evement du morphisme compos\'e $$\mathrm{Spec}\,\overline{k} \longrightarrow \mathrm{Spec}\,k \stackrel{Q}{\longrightarrow} \mathcal{V}$$ \`a $\mathcal{X}$. Par construction des places $v_g$, cette derni\`ere assertion est satisfaite. \end{proof}

\begin{coro} \label{schinzelschinzel} Si on admet l'hypoth\`ese de Schinzel \cite[hypoth\`ese (H)]{CTSD94} sur $\mathbf{Q}$, alors la version \cite[hypoth\`ese (H$_1$)]{CTSD94} de cette hypoth\`ese vaut sur tout corps de nombres $k$, avec la condition suppl\'ementaire que l'on peut choisir le param\`etre $\lambda$ dans un ensemble hilbertien $H \subseteq k = \mathbb{A}^1_k(k)$ donn\'e.
\end{coro}
\begin{proof}[D\'emonstration] Partons de polyn\^omes $P_i(t)$ pour $1 \leq i \leq n$ et d'\'el\'ements $\lambda_v \in k_v$ pour $v \in S$, o\`u $S \subseteq \Omega_k$ est fini, comme dans \cite[hypoth\`ese (H$_1$)]{CTSD94}. En agrandissant $S$ si n\'ecessaire, on peut supposer que l'ensemble hilbertien $H$ est associ\'e \`a un rev\^etement \'etale de $\mathcal{O}_{k,S}$-sch\'emas $\rho: \mathcal{X} \to \mathbb{A}^1_{\mathcal{O}_{k,S}}$, comme dans la proposition \ref{HilbertSchinzel}. 

Choisissons un $k$-point $P$ de coordonn\'ee $\mu$ tel que la fibre $\rho^{-1}(P)$ soit connexe. On trouve un ensemble fini $T \subseteq \Omega_k^{< \infty}$ avec $S \cap T = \emptyset$ et un voisinage $\mathcal{N}$ de $P$ dans la topologie induite par $T$ tel que $\mathcal{N} \subseteq H$. Prenons $\lambda_v = \mu \in k_v$ pour $v \in T$. En appliquant \cite[hypoth\`ese (H$_1$)]{CTSD94} aux polyn\^omes $P_i(t)$ pour $1 \leq i \leq n$ et aux $\lambda_v$ pour $v \in S \cup T$, on obtient maintenant le r\'esultat voulu.
\end{proof}
\subsection{Exemple d'application} L'int\'er\^et du Corollaire \ref{schinzelschinzel} ci-dessus dans le contexte de la m\'ethode des fibrations vient de \cite[Th\'eor\`eme 3.5.1]{Har}. Voici un simple exemple d'application dans ce sens, qui g\'en\'eralise \cite[Theorem 1.1.e]{CTSkoSD98}:

\begin{coro} \label{hasseoverhilbert} Soit $X$ une $k$-vari\'et\'e projective et lisse, munie d'un morphisme dominant $f: X \to \mathbb{P}^1_k$, telle que la condition suivante soit satisfaite: pour tout point ferm\'e $P$ de $\mathbb{P}^1_k$,  la fibre $X_P$ a une composante $Y_P$ de multiplicit\'e $1$ telle que la cl\^oture alg\'ebrique de $k(P)$ dans le corps de fonctions de $Y_P$ soit une extension ab\'elienne de $k(P)$. 

Supposons que le principe de Hasse et l'approximation faible valent pour les fibres de $f$ au-dessus d'un ensemble hilbertien de $\mathbb{P}^1_k$. Admettons l'hypoth\`ese de Schinzel. Alors $X(k)$ est dense dans $X(\mathbf{A}_k)^{\mathrm{Br}_{\mathrm{vert}}(X)}$.  \end{coro} 

Rappelons que le groupe de Brauer vertical $\mathrm{Br}_{\mathrm{vert}}(X)$ est le sous-groupe de $\mathrm{Br}(X)$ qui consiste des \'el\'ements dont la restriction \`a $\mathrm{Br}(X_\eta)$ provient de $\mathrm{Br}(\eta)$. 

La validit\'e de cet \'enonc\'e r\'esulte imm\'ediatement du corollaire \ref{schinzelschinzel} et de la preuve de \cite[Theorem 1.1]{CTSkoSD98}. Voici un exemple d'application concr\`ete de ce r\'esultat. 

\begin{prop} Soit $X$ un mod\`ele projectif et lisse d'une $k$-vari\'et\'e affine donn\'e par une \'equation de la forme $$x^2 - a(t) y^2 = P(t,z),$$ o\`u le polyn\^ome $P(t,z)$ est irr\'eductible de degr\'e $4$ vu comme \'el\'ement de $k(t)[z]$ et a la propri\'et\'e que $P(t_0,z)$ ne s'annule pas identiquement si $t_0$ est une racine de $a(t)$. Admettons l'hypoth\`ese de Schinzel. Alors $X(k)$ est dense dans $X(\mathbf{A}_k)^{\mathrm{Br}_\mathrm{vert}(X)}$. \end{prop}

\begin{proof}[D\'emonstration] La vari\'et\'e $X$ est fibr\'ee sur $\mathbb{P}^1_k$ via la projection sur $t$. Il est clair que l'hypoth\`ese sur les ``mauvaises'' fibres dans le Corollaire \ref{hasseoverhilbert} est satisfaite sous les conditions de l'\'enonc\'e. L'ensemble des valeurs $s \in k$ ayant la propri\'et\'e que le polyn\^ome $P(s,z) \in k[z]$ reste irr\'eductible est un ensemble Hilbertien dans $\mathbb{A}^1_k(k)$. Pour une telle valeur de $s$, la surface de Ch\^atelet donn\'ee par l'\'equation affine $x^2 - a(s) y^2 = P(s,z)$ satisfait au principe de Hasse et \`a l'approximation faible \cite{CTSanSD87a}. Ceci suffit pour appliquer le Corollaire \ref{hasseoverhilbert}. \end{proof}

\end{document}